\documentclass{article}
\usepackage{graphicx}
\usepackage[utf8]{inputenc}
\usepackage{amsfonts}
\usepackage{amsthm}
\usepackage{amsmath}
\usepackage{amssymb}
\usepackage{tikz}
\usepackage{tikz-cd}
\usepackage{parskip}
\usepackage{tikz}
\usepackage{biblatex}
\usepackage{subcaption}
\usetikzlibrary{patterns}

\title{Widths of Complements of Skeleta}
\date{}
\author{Elliot Gathercole}
\numberwithin{equation}{section}
\numberwithin{figure}{section}
\begin{document}

\newcommand{\Size}{\text{Size}}

\maketitle
\theoremstyle{plain}
\newtheorem{thm}{Theorem}

\newtheoremstyle{TheoremNum}
    {\topsep}{\topsep}              
    {\itshape}                      
    {}                              
    {\bfseries}                     
    {.}                             
    { }                             
    {\thmname{#1}\thmnote{ \bfseries #3}}
\theoremstyle{TheoremNum}
\newtheorem{thmn}{Theorem}

\theoremstyle{plain}
\newtheorem{prop}{Proposition}[section]
\newtheorem{cor}[prop]{Corollary}
\newtheorem{lem}[prop]{Lemma}

\theoremstyle{definition}
\newtheorem{defn}[prop]{Definition}
\newtheorem{ex}[prop]{Example}
\newtheorem{hyp}[prop]{Hypothesis}

\begin{abstract}
    We establish some sufficient conditions for the Lagrangian skeleton of the affine complement of an effective ample $\mathbb{Q}$-divisor in a smooth rationally connected projective variety to be a Lagrangian barrier in the sense of Biran, and establish bounds on the Gromov width of the complement of the skeleton. We particularly focus on hyperplane arrangements in projective space, where we obtain tight bounds in two dimensions when the divisor is a generic collection of at least three lines.
\end{abstract}

\section{Introduction}
\subsection{Background}
The problem of embedding balls in symplectic manifolds is much-studied: dating back to Gromov's famous non-squeezing theorem (\cite{Gro}), it is known that there are constraints on such embeddings which do not come simply from volume considerations. More precisely, we may consider the Gromov width, a quantitative invariant associated to a symplectic manifold, which records the largest possible ball that can be embedded.
\begin{defn}[Gromov Width]
Let $B_r\subset \mathbb{C}^n$ denote the standard ball of radius $r>0$, equipped with the standard symplectic form $\omega_{st}$. Recall that the Gromov width of a symplectic manifold $(M,\omega)$, denoted $W_G(M)$,  is the supremum over all $\pi r^2$ such that there exists a symplectic embedding
\begin{equation}
    \iota:(B_r,\omega_{st})\rightarrow (M,\omega).
\end{equation} 
\end{defn}

In \cite{Biran}, Biran introduced the phenomenon of Lagrangian barriers: subsets of a symplectic manifold which intersect all balls of a certain radius, which are isotropic CW-complexes (for example, Lagrangian submanifolds), which in particular have zero volume. These subsets give us examples of intersection phenomena which are symplectomorphism-invariant, that do not arise from purely topological or volume considerations. 
\begin{defn}[Barrier]
A closed subset $B\subset M$ of a symplectic manifold is a barrier if $W_G(M\setminus B)<W_G(M)$.
\end{defn}
In \cite{Biran}, Biran showed that in a K\"ahler manifold $(M,\omega)$, either symplectically aspherical or of complex dimension $\leq 3$, if $D\subset M$ is a smooth complex hypersurface Poincar\'e dual to $k[\omega]$ for some $k\in\mathbb{N}$, then the Lagrangian skeleton $L$ of the affine variety $M\setminus D$ satisfies $W_G(M\setminus L)\leq \frac{1}{k}$, so $L$ will be a barrier for sufficiently large $k$. This was later extended to all K\"ahler manifolds by Lu in \cite{Lu}.

This result works by showing that $M\setminus L$ is symplectomorphic to the normal disc bundle of $D$ of radius $\frac{1}{k}$, then using a fibrewise compactification of this normal bundle to construct holomorphic curves in the class of a fibre. Then, given a symplectic embedding $\iota:B_r\rightarrow M\setminus L$, Gromov-Witten invariants are used to show that these curves persist after deforming the almost-complex structure, to obtain a curve which is a complex subvariety of $B_r$, the existence of which implies that $\pi r^2<\frac{1}{k}$.

Another example of this phenomenon is \cite{Bre}, which shows that Lagrangian pinwheels in $\mathbb{CP}^2$ are barriers. 

\subsection{Symplectic Setting}
We seek to obtain analogous results to Biran in \cite{Biran} concerning the Lagrangian skeleta of complements of more complicated (i.e. non-smooth) divisors $D$, for instance, normal crossings divisors. In order to state our main theorem, we must first introduce some terminology.

Let $(M,\omega)$ be a compact symplectic manifold. We are interested in the case when $(M,\omega)$ is obtained from a projective variety, together with a choice of ample divisor. To model such a situation symplectically, we will consider a stratified symplectic subvariety $V$ in $M$, admitting a system of commuting Hamiltonians $r_v$ for $v\in V$, as described in \cite{Ga}.

\begin{defn}[Stratified Symplectic Subvariety]
A stratified symplectic subvariety in $M$ is a poset $V$ and a collection of disjoint embedded connected symplectic submanifolds $\mathring{D}_v\subset M$ of positive real codimension for $v\in V$, such that the closure of $\mathring{D}_v$ in $M$ is the union $\bigcup_{u\leq v}\mathring{D}_u$. 
We denote by $V_{\operatorname{top}}\subset V$ the subset corresponding to submanifolds of real codimension $2$.
\end{defn}

We will impose the following positivity conditions relating the symplectic form and first Chern class to $V$.
\begin{defn}[Cohomology Classes Supported on $V$]
Let $X=M\setminus \bigcup_{v\in V}\mathring{D}_v$. We define the classes supported on $V$, $C_+(V)\subset H^2(M,X;\mathbb{Q})$ as follows. For each $v\in V_{\operatorname{top}}$ we obtain a class $PD^{rel}([\mathring{D}_v])\in H^2(M,X;\mathbb{Q})$ given by intersection with the submanifold $\mathring{D}_v$. Let
\begin{equation}
    C_+(V)=\bigoplus_{v\in V_{\operatorname{top}}} \mathbb{Q}_{> 0}PD^{rel}([\mathring{D}_v]).
\end{equation}
\end{defn}
We will require that the classes $[\omega]$ and $c_1(TM)$ are supported on $V$, i.e.
\begin{hyp}
\begin{equation}
\label{positivity}
\begin{split}
[\omega]=\pmb{\kappa}|_{H_2(M)}, \quad & \pmb{\kappa}\in C_+(V)
\\2c_1=\pmb{\lambda}|_{H_2(M)}, \quad & \pmb{\lambda}\in C_+(V)
\end{split}
\end{equation}
\end{hyp}

\begin{defn}[System of Commuting Hamiltonians]
\label{sch_def}
A system of commuting Hamiltonians for $V$ is a choice of radial Hamiltonian for each $v\in V$: a real function $r_v$ defined on a neighbourhood of $\mathring{D}_v$, which generates a Hamiltonian $\mathbb{R}/\mathbb{Z}$-action, such that, possibly after shrinking these neighbourhoods, the circle action preserves $\mathring{D}_w$ whenever $v<w$, the functions $r_v$ commute pairwise wherever they are both defined, $r_v\geq 0$ everywhere, and the fixed locus of the circle action is exactly $r_v^{-1}(0)=\mathring{D}_v$. We can restrict to smaller neighbourhoods so that the positive level sets of the $r_v$ all intersect transversely.
\end{defn}
\begin{defn}[Weight of a radial Hamiltonian]
In the setting of Definition \ref{sch_def}, we denote by $w_v$ the total weight of the isotropy representation of the circle action generated by $r_v$ at any point in $\mathring{D}_v$.
\end{defn}

\begin{defn}[Action of a radial Hamiltonian]
\label{action_def}
To each radial Hamiltonian $r_v$, we can associate a class $\delta_v\in \pi_2(M,X)$, represented by a path of $\mathbb{R}/\mathbb{Z}$-orbits generated by $r_v$, starting at a fixed point in $\mathring{D}_v$ and ending on some orbit contained in $X$. 

From the lifts $\pmb{\kappa},\pmb{\lambda}$ of $[\omega]$ and $2c_1$, we obtain the following numbers for each $v\in V$: 
\begin{equation}
\begin{split}
\kappa_v&=\pmb{\kappa}(\delta_v),
\\\lambda_v&=\pmb{\lambda}(\delta_v).
\end{split}
\end{equation}
\end{defn}

To define the Lagrangian skeleton of $X$, we must first restrict  our choice of Liouville primitives to those which are adapted to our system of commuting Hamiltonians, and which, without loss of generality, represent the class $\pmb{\kappa}$.
\begin{defn}[Adapted Primitive]
    A Liouville primitive $\theta$ for $\omega|_X$ (so $\theta=d\omega|_X$) is adapted if it is invariant under the circle actions generated by $r_v$ for each $v$, locally near $\mathring{D}_v$.

    We say $\theta$ represents the class $\pmb{\kappa}$ if $[\omega,\theta]=\pmb{\kappa}$, i.e.
    \begin{equation}
        \pmb{\kappa}(u)=\int_u\omega-\int_{\partial u}\theta
    \end{equation}
    for all $u\in \pi_2(M,X)$.

    Note that there always exists an adapted primitive representing $\pmb{\kappa}$ by an averaging construction.
\end{defn}

Let $\theta\in\Omega^1(X)$ be an adapted primitive for $\omega|_X$, representing $\pmb{\kappa}$, and let $L$ denote the Lagrangian skeleton of $(X,\theta)$. Let $Z$ denote the Liouville vector field on $X$ dual to $\theta$. We can define a continuous function $\rho^0:M\rightarrow \mathbb{R}_{\geq 0}$, such that $\rho^0|_{\bigcup_{v\in V}\mathring{D}_v}=1$ and along flowlines of $Z$, $d\rho^0(Z)=\rho^0$.

\subsection{Algebro-geometric Setting}
Given an effective ample $\mathbb{Q}$-divisor $D$ in a projective variety $M$, we obtain a canonical choice of symplectic form $\omega$ coming from the associated embedding into projective space, and a lift $\pmb{\kappa}\in H^2(M,X;\mathbb{Q})$ of $[\omega]$ represented by $D$ under Poincar\'e duality. The underlying subset of $D$ can be given the structure of a stratified symplectic subvariety $V$ in the obvious way (so the top-dimensional strata correspond to irreducible components of $D$), which supports the class $\pmb{\kappa}$. Suppose the components of $D$ also support $c_1(TM)$ as an effective $\mathbb{Q}$-divisor. Then we similarly obtain a lift $\pmb{\lambda}$ of $2c_1(TM)$ supported on $V$. Under certain conditions on $D$, we can construct a system of commuting Hamiltonians for $V$.

If $D$ is an orthogonal symplectic crossings divisor, a system of commuting Hamiltonians can be obtained from a standard neighbourhood of $D$ (see \cite{TMZ}). In \cite{Ga} we show how to apply this model to a divisor $D$ which has singularities which are orthogonal symplectic crossings, or isolated quasihomogeneous singularities (and some other cases). In Section \ref{hyperplane_arr_construction}, we will give an explicit construction of such a model for any union of hyperplanes in $\mathbb{CP}^n$, the particular focus of this paper, which allows for degenerate configurations, and does not require the hyperplanes to be orthogonal.

In this setting, the affine complement $X$ is equipped with a standard (up to deformation) Liouville primitive $\theta$ for $\omega|_X$ which represents $\pmb{\kappa}$. Under good conditions on the system of commuting Hamiltonians, we may take $\theta$ to be adapted.

\subsection{Statement of Results}
Our results will give bounds on the Gromov width of the subsets $\{\rho^0> \sigma\}$. In particular, we will relate $W_G(\{\rho^0> \sigma\})$ to the quantities,
\begin{equation}
    \tilde{\sigma}_{crit}=\max_{v\in V}\left(\frac{\lambda_v-2w_v}{\lambda_v}\right),
    \label{sigma_crit}
\end{equation} 
\begin{equation}
\label{lambda_min}
\kappa_{\min}=\min\{\pmb{\kappa}(u)| u\in \pi_2(M,X), \pmb{\kappa}(u)>0\}\in\mathbb{Q}_{\geq 0},  
\end{equation}
and the minimal Chern number
\begin{equation}
    N_M=\min\{c_1(u)|u\in\pi_2(M), c_1(u)>0\}\in\mathbb{N},
\end{equation}
as well as a choice of class $A\in H_2(M)$ satisfying the following hypothesis.
\begin{hyp}
\label{rat_con}
    For any two points $x,y\in M$, and any $\omega$-compatible almost-complex structure $J$, there exists a $J$-holomorphic curve of genus $0$, representing the homology class $A\in H_2(M;\mathbb{Z})$, passing through $x$ and $y$.
\end{hyp}
For instance, it is sufficient that the appropriate genus-$0$ Gromov-Witten invariant is well-defined and not zero. For a projective variety, this is equivalent to being rationally connected in many cases (\cite{Tian_2012}).

\begin{thmn}[\ref{T:1}]
Suppose $A\in H_2(M)$ satisfies Hypothesis \ref{rat_con}. 

If $\tilde{\sigma}_{crit}\leq 0$, then
\begin{equation}
W_G(\{\rho^0> \sigma\})\leq \omega(A)-2\sigma N_M -(1-\sigma)\lceil1-\tilde{\sigma}_{crit}N_M\rceil\kappa_{\min}.
\end{equation}
\end{thmn}
If a class $A$ with suitable properties exists, we can establish sufficient conditions for the skeleton $L$ to be a barrier. For example,
\begin{thmn}[\ref{T:2}]
    Take $M=\mathbb{CP}^n$, and $\omega$ the Fubini-Study form , normalised so that $\omega(\mathbb{CP}^1)=1$. If $\tilde{\sigma}_{crit}\leq 0$, then $L\subset \mathbb{CP}^n$ is a Lagrangian barrier. Moreover,
    \begin{equation}
        W_G(\mathbb{CP}^n\setminus L)\leq 1-\lceil 1-\tilde{\sigma}_{crit}(n+1)\rceil \kappa_{\min}.
    \end{equation}
\end{thmn}

We can now obtain some specific examples of Lagrangian barriers $L\subset\mathbb{CP}^n$, together with bounds on the Gromov width of $\mathbb{CP}^n\setminus L$, as corollaries of Theorem \ref{proj_skel}, and compare them to existing results.

\begin{cor}[Generic Hyperplanes in $\mathbb{CP}^n$]
\label{generic_hyperplanes}
Let $M=\mathbb{CP}^n$, and let $D=H_1+...+H_\ell$ be a union of $\ell\geq n+1$ hyperplanes in general position. Let $L$ denote the skeleton of $X=M\setminus D$. Now, $\tilde{\sigma}_{crit}=\frac{n+1-\ell}{n+1}$, so Theorem \ref{proj_skel} tells us that 
\begin{equation}
    W_G(\mathbb{CP}^n\setminus L)\leq \frac{n}{\ell}.
    \label{generic_hyperplanes_equation}
\end{equation}
\end{cor}
In the case $\ell=n+1$, we can take $L$ to be the Clifford torus, and $D$ the toric boundary divisor, and Equation \eqref{generic_hyperplanes_equation} becomes an equality (a result already known by a variety of methods).

If we instead fix $n=2$, for any $\ell\geq 3$, if we take $H_1$, $H_2$ to be coordinate hyperplanes, we prove in Section \ref{lower_bounds} that Equation \eqref{generic_hyperplanes_equation} is an equality.

We can also apply this framework to some more singular divisors, for example, degenerate hyperplane configurations.
\begin{cor}[Degenerate Hyperplanes in $\mathbb{CP}^n$]
\label{deg_hyp}
Let $M=\mathbb{CP}^n$, and suppose $D=H_1+ ...+ H_\ell$ is an arbitrary union of $\ell$ hyperplanes. Let $L$ denote the skeleton of $X=M\setminus D$. Consider the quantity
\begin{equation}
m(D)=\min_{I\subset [\ell]}\frac{\operatorname{codim}_\mathbb{C}\bigcap_{i\in I}H_i}{\#I}.
\end{equation}
Note that $m(D)\leq 1$ with equality if and only if $D$ is normal crossings. Now, $\tilde{\sigma}_{crit}=\frac{n+1-\ell m(D)}{n+1}$, Theorem \ref{proj_skel} tells us that, if $\ell m(D)\geq n+1$, $L$ is a barrier and
\begin{equation}
    W_G(M\setminus L)\leq \frac{n}{\ell}+1-\frac{1}{\ell}\left\lceil \ell m(D) \right\rceil
\end{equation}
\end{cor}

For example, identify $\mathbb{CP}^n$ with $\{\sum_{i=0}^{n+1}z_i=0\}\subset \mathbb{CP}^{n+1}_z$, let $H_{i,j}=\{z_i=z_j\}\subset\mathbb{CP}^n$, and let $D=\sum_{i<j}H_{i,j}$ (the $A_{n+2}$ hyperplane arrangement). We have $\ell=\frac{(n+1)(n+2)}{2}$, and $m(D)=\frac{2}{n+1}$ (attained at the point $[1:...:1:-(n+1)]$, contained in $\frac{n(n+1)}{2}$ hyperplanes), so by Corollary \ref{deg_hyp}, the skeleton $L$ of $\mathbb{CP}^n\setminus D$ is a barrier.

\begin{cor}[Line configurations in $\mathbb{CP}^2$]
\label{deg_lines}
Let $M=\mathbb{CP}^2$, suppose $H_1,...,H_\ell$ are distinct projective lines, such that any point is contained in at most $k$ lines, and let $D=H_1+ ...+ H_\ell$. Let $L$ denote the skeleton of $X=M\setminus D$. Now, $\tilde{\sigma}_{crit}=\frac{3k-2\ell}{3k}$. Theorem \ref{proj_skel} tells us that, if $2\ell\geq 3k$, $L$ is a barrier (see Figure \ref{line_arrs}) and
\begin{equation}
    W_G(M\setminus L)\leq \frac{2}{\ell}+1-\frac{1}{\ell}\left\lceil \frac{2\ell}{k}\right\rceil.
\end{equation}
\end{cor}

\begin{figure}
\caption{Line arrangements: \ref{line_arrs:2} and \ref{line_arrs:4} satisfy the conditions of Corollary \ref{deg_lines}; \ref{line_arrs:1} and \ref{line_arrs:3} do not.}
\begin{subfigure}{0.24\textwidth}
\begin{tikzpicture}[scale=0.9]
    \draw (0,0)--(2,3);
    \draw (1,3)--(3,0);
    \draw (0,0.5)--(3,0.5);
    \draw (1.5,0)--(1.5,3);
\end{tikzpicture}
\caption{$(\ell,k)=(4,3)$}
\label{line_arrs:1}
\end{subfigure}
\begin{subfigure}{0.24\textwidth}
\begin{tikzpicture}[scale=0.9]
    \draw (0,0)--(2,3);
    \draw (1,3)--(3,0);
    \draw (0,0.5)--(3,0.5);
    \draw (1.5,0)--(1.5,3);
    \draw (0,0.25)--(3,2.5);
\end{tikzpicture}
\caption{$(\ell,k)=(5,3)$}
\label{line_arrs:2}
\end{subfigure}
\begin{subfigure}{0.24\textwidth}
\begin{tikzpicture}[scale=0.9]
    \draw (0,0)--(2,3);
    \draw (1,3)--(3,0);
    \draw (0,0.5)--(3,0.5);
    \draw (1,0)--(5/3,3);
    \draw (2,0)--(4/3,3);
\end{tikzpicture}
\caption{$(\ell,k)=(5,4)$}
\label{line_arrs:3}
\end{subfigure}
\begin{subfigure}{0.24\textwidth}
\begin{tikzpicture}[scale=0.9]
    \draw (0,0)--(2,3);
    \draw (1,3)--(3,0);
    \draw (0,0.5)--(3,0.5);
    \draw (1,0)--(5/3,3);
    \draw (2,0)--(4/3,3);
    \draw (0,0.25)--(3,2.5);
\end{tikzpicture}
\caption{$(\ell,k)=(6,4)$}
\label{line_arrs:4}
\end{subfigure}
\label{line_arrs}
\end{figure}
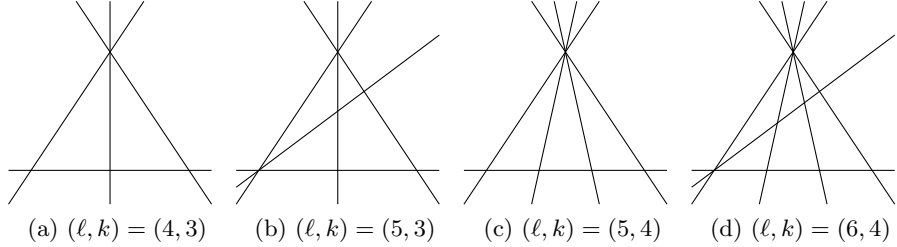

By contrast, the following example fails to give a tight bound on the Gromov width in any case.
\begin{cor}[Smooth Divisor in $\mathbb{CP}^n$]
Let  $M=\mathbb{CP}^n$, and suppose $D$ is a smooth hypersurface of degree $d\geq n+1$. Let $L$ denote the skeleton of $X=M\setminus D$. Now, $\tilde{\sigma}_{crit}=\frac{n+1-d}{n+1}$, so Theorem \ref{proj_skel} tells us that 
\begin{equation}
W_G(\mathbb{CP}^n\setminus L)\leq \frac{n}{d}.
\end{equation}
\end{cor}
Note that this is a much weaker result than that obtained by Biran in \cite{Biran}, which gives a tight bound, $W_G(\mathbb{CP}^n\setminus L)\leq \frac{1}{d}$.

\section{Hyperplane Arrangements}
\label{hyperplane_arr_construction}
We now give a short construction of a system of commuting Hamiltonians for a stratification of an arbitrary finite union of hyperplanes in $\mathbb{CP}^n$, endowed with the Fubini-Study form. Let $H_1,...,H_\ell$ be distinct hyperplanes in $\mathbb{CP}^{n}$, and let $\tilde{H}_i$ denote the lift of $H_i$ to $\mathbb{C}^{n+1}$. Let $V$ denote the set of non-zero $\mathbb{C}$-linear subspaces of $\mathbb{C}^{n+1}$ which are of the form $\cap_{i\in I}\tilde{H}_i$ for some $I\subset [\ell]$. The relation of containment gives a natural partial order $\leq$ on $V$. For each $v\in V$, let $D_v$ denote the projective subspace of $\mathbb{CP}^n$ obtained as the image of $v$ under the quotient map (so $D_{\tilde{H}_i}=H_i$), and $\mathring{D}_v=D_v\setminus \bigcup_{u<v}D_u$. This makes $V$ a stratified symplectic subvariety.

For each linear subspace $v\in V$, consider the unitary representation $\Phi_v^t$ of $\mathbb{R}/\mathbb{Z}$ on $\mathbb{C}^{n+1}$ which has the form $\Phi^t_v=(id,e^{2\pi it}id)$ with respect to the orthogonal splitting $\mathbb{C}^{n+1}=v\oplus v^\perp$ (i.e. rotation around $v$ with period $1$). This descends to an action $\phi_v^t$ on $\mathbb{CP}^n$ by K\"ahler isometries. This is a Hamiltonian circle action, generated by $r_v:\mathbb{CP}^n\rightarrow \mathbb{R}_{\geq 0}$, where $r_v([z])=\frac{||\pi_v^\perp(z)||^2}{||z||^2}$, where $\pi_v^\perp$ denotes projection onto the orthogonal complement of $v$. This action is free on a punctured neighbourhood of $D_v$, and fixes $D_v$ itself.

Now suppose $u<v$. Consider the orthogonal splitting $\mathbb{C}^{n+1}=u\oplus (u^\perp\cap v)\oplus v^\perp$. The representations $\Phi_u^t$ and $\Phi_v^t$ are both diagonal with respect to this splitting, so the actions $\phi_u^t$ and $\phi_v^t$ commute, and $\phi_u^t$ preserves $D_v$. The functions $r_v$ for $v\in V$ are therefore a system of commuting Hamiltonians.

For $v\in V$, the weight of $r_v$ is $w_v=\dim_{\mathbb{C}}v^\perp=\operatorname{codim}_{\mathbb{C}}v$. Now consider an effective $\mathbb{Q}$-divisor of the form
\begin{equation}
D=\sum_{i=1}^\ell\lambda_i H_i.
\end{equation}
The affine variety $X=\mathbb{CP}^n\setminus D$ has a canonical (up to deformation) exact symplectic structure representing the class Poincar\'e dual to $D$. For $v\in V$, let $\lambda_v=\sum_{i| v\leq \tilde{H}_i}\lambda_i$. For topological reasons, we see that the action of $r_v$ with respect to this class is $\lambda_v$.

Because the circle actions $\phi_v^t$ locally preserve the divisor $D$, the Liouville vector field on $X$ is outward pointing with respect to $r_v$, so $\theta$ is equivalent to an adapted primitive.

\section{Properties of $(Y^R,\alpha)$}
\subsection{Construction of $(Y^R,\alpha)$}
As described in \cite{Ga}, there exists a smoothing family of hypersurfaces $Y^R\subset X$ for $R\in (0,R_0)$, the interiors of which exhaust $X$ as $R\to 0$, with the following properties. There exists a family $\{\mathring{U}_I:I\subset V\}$ of open subsets of $M$ covering $|V|$, and a smooth real function $H$ defined on a neighbourhood of $Y^R$ such that $Y^R=H^{-1}(1)$ and $H$ has the following form:
\begin{equation}
H|_{\mathring{U}_I}=\sum_{v\in I}q_v(r_v),
\label{H_standard_form}
\end{equation}
where for each $v\in V$, $q_v$ is a smooth function satisfying $q_v\geq 0$, $q_v'\leq 0$ with equality only if $q_v=0$, and $q''_v\geq 0$. Furthermore, there is a continuous, increasing function $\epsilon:[0,R_0)\rightarrow \mathbb{R}_{\geq 0}$ such that $\epsilon(0)=0$ and for all $v\in V$, $q_v(x)=0$ for $x\geq\epsilon(R)$. We will use these properties to compute the Reeb flow on $Y^R$.

We can define a family of continuous functions $\rho^R:M\rightarrow \mathbb{R}_{\geq 0}$, smooth on $M\setminus L$, such that $\rho^R|_{Y^R}=1$ and $\iota_Zd\rho=\rho$ (by construction $\rho^R\to \rho^0$ as $R\to 0$). Letting $\alpha$ denote the pullback of $\theta$ to the hypersurface $Y^R$, the region $\{0<\rho^R\leq 1\}$ is symplectomorphic to the symplectisation $\left((0,1]_\rho\times Y^R,d(\rho\alpha)\right)$.

\subsection{Classification of Reeb Orbits}
Let $\gamma$ be a Reeb orbit in $(Y^R,\alpha)$ of period $S$, so $\gamma(t)=\phi^t(x)$ for some $x\in Y$, where $\phi^t$ denotes the time $S$ Reeb flow.
Let 
\begin{equation}
    I(\gamma)=\{v\in V| x\in U_v, q_v'(r_v(x))\neq 0\}.
\end{equation}
This makes sense because $r_v$ is constant along $\gamma$ whenever $\gamma\subset U_v$.

For $I\subset V$, the Hamiltonians $(r_v:v\in I)$ generate a $\mathbb{R}^I\mathbb{Z}^I$-action on $Y\cap \mathring{U}_I$. Let $O(\gamma)$ denote the $\mathbb{R}^I/\mathbb{Z}^I$-orbit containing $\gamma$.
For $y\in \mathring{U}_I\cap Y$, denote the period lattice by $\Lambda_I(y)=\{a\in \mathbb{R}^I| \phi_{r_I}^{a}(y)=y\}$, where $\phi^a_{r_I}:M\rightarrow M$ is the action of $a\in\mathbb{R}^I/\mathbb{Z}^I$. We have that $\mathbb{Z}^I\subset \Lambda_I(y)$. For fixed $I$, $\Lambda_I(y)$ is upper semi-continuous in $y$. Because the radial Hamiltonians $r_v$ cut out distinct submanifolds, we may restrict to sufficiently small neighbourhoods so that the level sets are transverse, and $\Lambda_I(y)$ is therefore genuinely a lattice.

For $v\in V$, define a function $a_v:Y\rightarrow \mathbb{R}_{\geq 0}$ by letting
\begin{equation}
a_v|_{U_v\cap Y}=\frac{q_v'(r_v)}{\theta(X_H)}
\end{equation}
and letting $a_v|_{Y\setminus U_v}=0$. By construction, $a_v$ is smooth and $a_v(\gamma)\neq 0$ if and only if $v\in I(\gamma)$ (we can write $a_v(\gamma)$ because $a_v$ is constant along $\gamma$).

Along $\gamma$, the Reeb vector field is given by
\begin{equation}
\label{Reeb}
R=\sum_{v\in I(\gamma)}a_v(\gamma)X_{r_v},
\end{equation}
so $a(\gamma)\in S^{-1}\Lambda_{I(\gamma)}(x)$. 

\subsection{Preferred Trivialisations}

\begin{defn}
An $r_v$-equivariant Darboux chart is an open subset $U\subset M$, which is a union of Hamiltonian orbits generated by $r_v$, with a symplectic embedding $\varphi:U\rightarrow \mathbb{C}^n$, such that $\varphi(U)$ is star-shaped with centre $0$, and $\phi$ intertwines the Hamiltonian $\mathbb{R}/\mathbb{Z}$-action generated by $r_v$ with a unitary representation $\mathbb{R}/\mathbb{Z}\rightarrow SU(n)$.    
\end{defn}
Note that the existence of $r_v$-equivariant Darboux charts covering some neighbourhood of $\mathring{D}_v$ is a consequence of the ($\mathbb{R}/\mathbb{Z}$-equivariant) Darboux' theorem (see \cite{GS}, Theorem 4.1). By restricting the range of the Hamiltonians $r_v$, we may assume that all orbits of $X_{r_v}$ are covered by equivariant Darboux charts which are arbitrarily small with respect to some fixed metric on $M$.

\begin{defn}[Outer Caps]
\label{outer_caps}
For any $S$-periodic Reeb orbit $\gamma$ in $(Y^R,\alpha)$, an outer cap is any representative of the class in $\delta\in\pi_2(M,\gamma)$ with boundary $[\gamma]$ defined as follows. We first define a rational class
\begin{equation}
    \delta=\sum_{v\in I(\gamma)}Sa_v(\gamma)\delta_v\in \pi_2(M,O(\gamma))\otimes\mathbb{Q},
\end{equation}
then show that $\delta\in\pi_2(M,O(\gamma))$.

By assumption on the diameter of the equivariant Darboux charts, we may assume that for all $v\in I(\gamma)$, we can represent $\delta_v$ by a disc contained in $B$, where $B$ is an embedded ball (we can choose $B$ to be a ball of diameter $<inj(M)$ where $inj(M)$ denotes the radius of injectivity of $M$ with respect to some metric). Because the boundary map $\pi_2(B,O(\gamma))\rightarrow \pi_1(O(\gamma))$ is an isomorphism, and the image of $\delta$ is the integral class $[\gamma]$ by construction, we have that $\delta\in\pi_2(M,O(\gamma))$ as required.

We denote by $u_{out}(\gamma)$ some representative of $\delta$ which bounds $\gamma$.
\end{defn}

Recall that the relative first Chern number $c_1^\tau(u)$ counts the number of zeros of a smooth section of the complex line bundle $\Lambda^nTM$ (with respect to some $\omega$-compatible almost-complex structure on $M$), which is constant along $\partial u$ with respect to the trivialisation $\tau$ (see Definition 5.1 of \cite{Wendl}). These numbers will be used to compare Conley-Zehnder indices with independently of trivialisations. 

\begin{defn}[Outer Trivialisations]
For any $S$-periodic Reeb orbit $\gamma$ in $(Y^R,\alpha)$, we denote by $\tau_{out}$ the equivalence class of trivialisations of $\gamma^*\xi$ or $\gamma^*TM$ such that the induced trivialisations of $\Lambda^n_\mathbb{C}TM$ extend over $u_{out}(\gamma)$, i.e. $c_1^{\tau_{out}}(u_{out}(\gamma))=0$. 
\end{defn}

\begin{lem}[Area of Outer Caps]
\label{outer_caps_action}
For any $S$-periodic Reeb orbit $\gamma$ in $(Y^R,\alpha)$,
\begin{equation}
    0\leq\omega(u_{out}(\gamma))\leq C\epsilon(R)\pmb{\kappa}(u_{out}(\gamma)),
\end{equation}
where $C>0$ is a constant depending only on $\kappa_{\min}$.
\end{lem}
\begin{proof}
Let $C=\frac{1}{\kappa_{\min}}$.

Because $O(\gamma)$ is isotropic, we obtain a $\mathbb{Q}$-linear map, $\omega:\pi_2(M,O(\gamma))\otimes\mathbb{Q}\rightarrow \mathbb{R}$. Let $u_v$ denote a disc which is a path of $\mathbb{R}/\mathbb{Z}$-orbits generated by $r_v$, starting at a point on $\mathring{D}_v$ and ending on $O(\gamma)$ (so $u_v$ is a representative of $\delta_v$). By linearity it suffices to show that $0\leq \omega(u_v)\leq C\epsilon(R)\kappa_v$.

Because $u_v$ is represented by a union of Hamiltonian orbits generated by $r_v$, which is constant on $O(\gamma)$, $\omega(u_v)=r_v(\gamma)\leq \epsilon(R)$, as required.
\end{proof}

\begin{defn}(Conley-Zehnder Index)
Let $A:[0,S]\rightarrow Sp(2n)$, be a path of symplectic matrices such that $A(0)=id$. We say $A$ is non-degenerate if $det(A(S)-id)\neq 0$. The Conley-Zehnder index is an integer associated any such non-degenerate path, which we will denote $i(A)$. There are several ways of extending this definition to degenerate paths (i.e. those such that $det(A(S)-id)=0$). We will use the version defined by Long in \cite{Long_index}. This version has the property that whenever $B:[0,S]\rightarrow Sp(2n)$ is sufficiently $C^0$-close to $A$ and $\det(B(S)-id)\neq 0$, $i(B)-i(A)\in [0,\dim\ker(A(S)-id)]$. 

Let $\xi$ be a contact distribution on a $2n-1$-manifold $Y$, and let $\phi^t$ denote the time-$t$ Reeb flow. To any $S$-periodic Reeb orbit $\gamma$, equipped with a trivialisation $\tau$ of $\gamma^*\xi$, we obtain a path in $Sp(2n-2)$, given by
\begin{equation}
    A(t)=\tau(t)^{-1}\circ (D\phi^t)_{\gamma(0)}|_\xi\circ \tau(0)
\end{equation}
and we can assign an integer, the Conley-Zehnder index of the Reeb orbit, given by
\begin{equation}
    \mu_{CZ}^\tau(\gamma)=i(A).
\end{equation}
When $\gamma$ is non-degenerate, this coincides with the notation used by Wendl in Definition 3.31 of \cite{Wendl}, which appears in the virtual dimension formula for holomorphic curves asymptotic to $\gamma$.
\end{defn}

\begin{lem}
\label{integral_orbit_index}
    For any $S$-periodic Reeb orbit $\gamma$ in $(Y^R,\alpha)$, if $Sa(\gamma)\in\mathbb{N}^{I(\gamma)}$, then
    \begin{equation}
        \mu_{CZ}^{\tau_{out}}(\gamma)+2S\sum_{v\in I(\gamma)}a_v(\gamma)w_v\in [-2n+2,2n-2].
    \end{equation}
\end{lem}
\begin{proof}
By hypothesis, $\gamma$ is a $1$-periodic orbit of the Hamiltonian $\mathbb{R}/\mathbb{Z}$-action generated by
\begin{equation}
G=S\sum_{v\in I(\gamma)}a_v(\gamma)r_v.
\end{equation}
Let $x=\gamma(0)$. Let $\tau_{eq}$ denote any trivialisation of $\gamma^*\xi$ arising from $D\phi^t_G$, i.e. $\tau_{eq}(t)=(D\phi_G^{S^{-1}t})_x|_\xi^{-1}:\xi_{\gamma(t)}\rightarrow \xi_x$, where we identify $\xi_x\cong\mathbb{C}^{n-1}$ in a fixed way.

First, we calculate $\mu_{CZ}^{\tau_{eq}}(\gamma)$. Let $A(t)=D(\phi^{St}\circ \phi_G^{-t})_x|_\xi$, so $\mu_{CZ}^{\tau_{eq}}(\gamma)=i(A)$, where we view $A$ as a map $[0,1]\rightarrow Sp(\xi_x)\cong Sp(2n-2)$.

The Lie algebra element $\dot{A}=\frac{dA}{dt}(0)$ is given by
\begin{equation}
\label{lie_element}
    \dot{A}=S\sum_{v\in I(\gamma)}\left(da_v\otimes X_{r_v}\right)_x|_\xi\in Sp(\xi_x)
\end{equation}
which satisfies $\dot{A}^2=0$ because for each $u,v\in I(\gamma)$, $\{a_u,r_v\}=0$, so $A(t)=id+t\dot{A}$. It follows from homotopy invariance that $i(A|_{[0,t]})$ is independent of $t\in (0,1]$. Letting $t\to 0$, we see that $i(A)\in [-2n+2,2n-2]$ because $A|_{[0,t]}$ becomes arbitrarily close to the constant path $id$.

Because $\mu_{CZ}^{\tau_{out}}(\gamma)=\mu_{CZ}^{\tau_{eq}}(\gamma)-2c_1^{\tau_{eq}}(u_{out}(\gamma))$, it suffices to show that $c_1^{\tau_{eq}}(u_{out}(\gamma))=S\sum_{v\in I(\gamma)}a_v(\gamma)w_v$.  First, we can extend $\tau_{eq}$ to a $\mathbb{R}^{I(\gamma)}/\mathbb{Z}^{I(\gamma)}$-invariant trivialisation of $\xi|_{O(\gamma)}$. The cap $u_{out}(\gamma)$ is homotopic in $O(\gamma)$ to a concatenation of discs $Sa_v(\gamma)u_v$ for $v\in I(\gamma)$ where $u_v$ is as in Definition \ref{outer_caps}, so
\begin{equation}
    c_1^{\tau_{eq}}(u_{out})=S\sum_{v\in I(\gamma)}a_v(\gamma)c_1^{\tau_{eq}}(u_v).
\end{equation}
The trivialisation $\tau_{eq}$ is equivalent as a trivialisation of $\gamma_v^*\xi$, where $\gamma_v=\partial u_v$, to that induced by an $r_v$-equivariant Darboux chart, so $c_1^{\tau_{eq}}(u_v)=w_v$, as required.
\end{proof}

\begin{lem}
\label{index_estimate}
    For any $S$-periodic Reeb orbit $\gamma$ in $(Y^R,\alpha)$, 
    \begin{equation}
    \mu_{CZ}^{\tau_{out}}(\gamma)+\dim\ker ((D\phi^1)_{\gamma(0)}|_\xi-id)\leq 2S\sum_{v\in I(\gamma)}a_v(\gamma)w_v+n-1.
    \end{equation}
\end{lem}

\begin{proof}
    First we will compute the mean index of $\gamma$, defined as follows:
    \begin{equation}
        \hat{\mu}_{CZ}^{\tau_{out}}(\gamma)=\lim_{m\to \infty}\frac{\mu_{CZ}^{\tau_{out}}(\gamma^m)}{m},
    \end{equation}
    where $\gamma^m$ is the $m$-th iterate of $\gamma$ (the corresponding $mS$-periodic Reeb orbit). Because a trivialisation in the class $\tau_{out}$ for $\gamma^*\xi$ induces a trivialisation for $(\gamma^m)^*\xi$ also in the class $\tau_{out}$, this limit is a well-defined real number, the mean index of the corresponding path in $Sp(2n-2)$ as defined in \cite{Long_iteration}.
    
    We can compute it by considering only those $m\in\mathbb{N}$ such that $mSa(\gamma)\in\mathbb{N}^{I(\gamma)}$. By Lemma \ref{integral_orbit_index}, 
    \begin{equation}
    \label{mean_index}
        \hat{\mu}_{CZ}^{\tau_{out}}(\gamma)=2S\sum_{v\in I(\gamma)}a_v(\gamma)w_v.
    \end{equation}

    By Theorem 4 of \cite{Long_iteration},
    \begin{equation}
    \mu_{CZ}^{\tau_{out}}(\gamma)+\dim\ker ((D\phi^1)_x|_\xi-id)\leq \hat{\mu}_{CZ}^{\tau_{out}}(\gamma)+n-1
    \end{equation}
    (there is also a lower bound that we will not use). Combining this upper bound with Equation \eqref{mean_index} gives the required result.
\end{proof}

\subsection{Perturbing $\alpha$}

\begin{defn}[Nearby Reeb orbits]
    Let $f:Y^R\rightarrow \mathbb{R}$. Let $\gamma,\tilde{\gamma}$ be $S,\tilde{S}$-periodic Reeb orbits in $(Y^R,\alpha)$, $(Y,e^f\alpha)$ respectively.

    We say $\tilde{\gamma}$ is an $\epsilon$-perturbation of ${\gamma}$ if there exist trivialisations $\tau$, $\tilde{\tau}$ for $\gamma^*\xi$, $\tilde{\gamma}^*\xi$ respectively, and a cylinder $\Sigma$ in $Y$ such that 
    \begin{itemize}
    \item $|S-\tilde{S}|<\epsilon$,
    \item $\partial \Sigma=\gamma-\tilde{\gamma}$,
    \item $\mu_{CZ}^\tau(\gamma)-\mu_{CZ}^{\tilde{\tau}}(\tilde{\gamma})\in [0,\dim\ker((D\phi^t)_{\gamma(0)}|_\xi-id)]$,
    \item  $c_1^{\tau,\tilde{\tau}}(\Sigma)=0$.
    \end{itemize}
    Although we are not requiring $\gamma,\tilde{\gamma}$ to be close as maps, this will be the case when we show the existence of such pairs.
\end{defn}
\begin{lem}
\label{single_orbit_perturbation}
Fix a metric on $Y^R$, and $\epsilon>0$. Let $\gamma$ be an $S$-periodic Reeb orbit in $(Y^R,\alpha)$. There exists $\delta>0$ such that whenever $||f||_{C^\infty}<\delta$, $|\tilde{S}-S|<\delta$, $\tilde{\gamma}$ is a non-degenerate $\tilde{S}$-periodic Reeb orbit in $(Y,e^f\alpha)$, and $d(\gamma(0),\tilde{\gamma}(0))<\delta$, then $\tilde{\gamma}$ is an $\epsilon$-perturbation of $\gamma$.
\end{lem}

\begin{proof}
    Suppose that $f$, the points $x=\gamma(0)$, $y=\tilde{\gamma}(0)$, and $\tilde{S}\in (S-\delta,S+\delta)$ satisfy the hypotheses. We may assume $\delta<\epsilon$.
    Let $\tilde{\phi}^t$ denote the time-$t$ Reeb flow on $(Y,e^f\alpha)$. Fix a tubular neighbourhood $\nu$ for $\gamma$, and a smooth trivialisation $\sigma$ of $\nu^*\xi$. 
    
    For sufficiently small $\delta$, the maps $[0,S]\times Y\rightarrow Y$ given by $(t,x)\mapsto \phi^t(x)$ and $(t,y)\mapsto \tilde{\phi}^{S\tilde{S}^{-1}t}(y)$ will be arbitrarily $C^\infty$-close. In particular, we can ensure that the map $\tilde{\gamma}$ lifts to a section of the normal bundle in $\nu$. Let $\tau$ and $\tilde{\tau}$ denote the restriction of $\sigma$ to $\gamma$ and $\tilde{\gamma}$ respectively, and $\Sigma$ any cylinder in $\nu$ connecting $\gamma$ to $\tilde{\gamma}$. Because $\sigma$ extends over $\Sigma$, $c_1^{\tau,\tilde{\tau}}(\Sigma)=0$. 
    
    From the trivialisation $\sigma$ we obtain paths
    \begin{equation}
    \begin{split}
    A(t)&=\tau(t)^{-1}\circ (D\phi^t)_x|_\xi\circ \tau(0)
    \\\tilde{A}(t)&=\tilde{\tau}(\tilde{S}S^{-1}t)^{-1}\circ (D\tilde{\phi}^{S\tilde{S}^{-1}t})_y|_\xi\circ \tilde{\tau}(0)
    \end{split}
    \end{equation}
    based at $id$ in $Sp(2n-2)$ such that $\mu_{CZ}^\tau(\gamma)=i(A)$ and $\mu_{CZ}^{\tilde{\tau}}(\tilde{\gamma})=i(\tilde{A})$. For small enough $\delta$, $\tilde{A}$ will, in particular, be arbitrarily $C^0$-close to $A$, so we can ensure that $i(\tilde{A})\in [i(A),i(A)+\dim\ker(A(S)-id)]$, whenever $\det(\tilde{A}(\tilde{S})-id)\neq 0$.
\end{proof}

\begin{lem}
\label{fixed_period_perturbation}
Fix a metric on $Y$, and $S>0$. There exists $\delta>0$ such that whenever $||f||_{C^\infty}<\delta$ and $|\tilde{S}-S|<\delta$, for every non-degenerate $\tilde{S}$-periodic Reeb orbit $\tilde{\gamma}$ in $(Y,e^f\alpha)$, there exists an $S$-periodic Reeb orbit $\gamma$ in $(Y^R,\alpha)$ such that $\tilde{\gamma}$ is an $\epsilon$-perturbation of $\gamma$.
\end{lem}
\begin{proof}
    For each $S$-periodic Reeb orbit $\gamma$ in $(Y^R,\alpha)$, let $\delta_\gamma$ be a constant satisfying Lemma \ref{single_orbit_perturbation}, and let $U_\gamma=\{y\in Y| d(y,\gamma(0))<\delta_\gamma\}$. The sets $U_\gamma$ form an open cover of the set $\{\phi^S(x)=x\}\subset Y$. We can choose a finite subcover indexed by $\Gamma$.

    By compactness there exists $\delta'>0$ such that for all $y\in Y$, either $d(y,\tilde{\phi}^{\tilde{S}}(y))>0$ whenever $|\tilde{S}-S|<\delta'$ and $||f||_{C^\infty}<\delta'$, or $y\in U_\gamma$ for some $\gamma\in\Gamma$. Let $\delta=\min(\delta',\min_{\gamma\in\Gamma}\delta_\gamma)$.

    Now suppose $||f||_{C^\infty}<\delta$, $\tilde{\gamma}$ is a non-degenerate $\tilde{S}$-periodic Reeb orbit in $(Y,e^f\alpha)$ with $\tilde{\gamma}(0)=y$, and $|\tilde{S}-S|<\delta$. Then by assumption, since $\delta<\delta'$, $y\in U_\gamma$ for some $\gamma\in\Gamma$. Because $\delta<\delta_\gamma$, $\tilde{\gamma}$ is an $\epsilon$-perturbation of $\gamma$.
\end{proof}

\begin{lem}
\label{bounded_period_perturbation}
    Fix a metric on $Y$, and $S_{\max}>0$. There exists $\delta>0$ such that whenever $||f||_{C^\infty}<\delta$ and $\tilde{S}<S_{\max}$, for any non-degenerate $\tilde{S}$-periodic Reeb orbit $\tilde{\gamma}$ in $(Y,e^f\alpha)$, there exists an $S$-periodic Reeb orbit $\gamma$ in $(Y^R,\alpha)$ such that $\tilde{\gamma}$ is an $\epsilon$-perturbation of $\gamma$ and $S<S_{\max}$.
\end{lem}
\begin{proof}
By Lemma \ref{fixed_period_perturbation}, for any $S\in [0,S_{\max}]$, there exists $\delta_S>0$ such that whenever $||f||_{C^\infty}<\delta_S$, $|S-\tilde{S}|<\delta_S$, and $\tilde{\gamma}$ is an $\tilde{S}$-periodic Reeb orbit in $(Y,e^f\alpha)$, there exists an $S$-periodic Reeb orbit $\gamma$ in $(Y^R,\alpha)$ such that $\tilde{\gamma}$ is an $\epsilon$-perturbation of $\gamma$.

By compactness, there exist $S_0<S_1<...<S_r$ for some $r\in\mathbb{N}$ such that the intervals $(S_i-\delta_{S_i},S_i+\delta_{S_i})$ for $i=1,...,r$ cover $[0,S_{\max}]$. Let $\delta=\min_{i=1}^r\delta_{S_i}$.

Suppose $\tilde{\gamma}$ is a non-degenerate $\tilde{S}$-periodic Reeb orbit in $(Y,e^f\alpha)$ with $\tilde{S}<S_{\max}$, and $||f||_{C^\infty}<\delta$. Now for some $i$, $|\tilde{S}-S_i|<\delta_{S_i}$, and since $||f||_{C^\infty}<\delta_{S_i}$, there exists an $S_i$-periodic Reeb orbit $\gamma$ in $(Y^R,\alpha)$ such that $\tilde{\gamma}$ is an $\epsilon$-perturbation of $\gamma$.
\end{proof}

\section{Neck Stretching Argument}
\begin{defn}
\label{perturbed_surface}
Given any smooth function $f:Y^R\rightarrow (-\infty,0]$, we obtain a contact-type hypersurface $Y^R_f$ by letting
\begin{equation}
    Y^R_f=\{(y,\rho)\in (0,1]_\rho\times Y^R\subset M\setminus L\; |\; \rho=e^{f(y)}\}.
\end{equation}
There is a natural map $p_f:Y^R\rightarrow Y^R_f$ given by $p_f(y)=(e^{f(y)},y)$, which is homotopic to the inclusion $Y^R\subset M\setminus L$. The restriction $\alpha_f=\theta|_{Y^R_f}$ is a contact form, and $p_f^*\alpha_f=e^f\alpha$.
\end{defn}

\begin{lem}
\label{neck_stretch}
Let $B_r\subset \mathbb{C}^n$ denote the standard ball of radius $r>0$. Suppose there exists a symplectic embedding $\iota:B_r\rightarrow \{\rho^0>\sigma\}$, for some $\sigma\in (0,1)$. Let $x=\iota(0)$, and fix $y\in L$. Let $f:Y^R\rightarrow (-\infty,0]$. Suppose all Reeb orbits in $(Y^R,e^f\alpha)$ are non-degenerate. Suppose $A\in H_2(M;\mathbb{Z})$ satisfies Hypothesis \ref{rat_con}.

There exists an almost-complex structure $J_0$ on $M$ suitable for neck-stretching such that the following holds. Let $(M_\infty,J_\infty)$ denote the split almost-complex manifold obtained by splitting $(M,J_0)$ along the contact-type hypersurface $Y_{\log\sigma+f}^R$. This consists of a bottom level $\hat{X}$ which is the completion of $(X,\theta)$, finitely many levels symplectomorphic to the symplectisation of $(Y^R,\alpha)$, and a top level symplectomorphic to $M\setminus L$, equipped with certain compatible almost-complex structures.

There is a $J_\infty$-holomorphic building $u$ in $(M_\infty,J_\infty)$ such that
\begin{itemize}
    \item $u$ is genus-$0$
    \item $[\overline{u}]=A$
    \item $ev(u)=(x,y)$
    \item $\iota^{-1}(u_+)$ is a holomorphic curve in $B_r$
    \item $u_-$ is (possibly a multiple cover of) a Fredholm-regular curve $\tilde{u}_-$,
    \item $\tilde{u}_-$ is a regular point for the evaluation map $ev$.
\end{itemize}
where $u_+$ is the component in $M\setminus L$ and $u_-$ the component in $\hat{X}$, which carry the marked points.
\end{lem}
\begin{proof}
Let $V=Y^R_{\log\sigma+f}$ and $\beta=\alpha_{\log\sigma+f}$.
The contact-type hypersurface $V$ splits $M$ into two manifolds with boundary: $M_+$, the convex part containing $B_r$, and $M_-$, a Liouville subdomain of $X$. Let $V_\pm$ denote the boundary of $M_\pm$, which is a copy of $V$. As in Section 2.7 of \cite{CM}, we can define a family of symplectic manifolds
\begin{equation}
M_\tau=M_+\cup_{V_+}[-\tau,0]\times V \cup_{V_-}M_-,
\end{equation}
where we identify $V_+$ with $\{0\}\times V$ and $V_-$ with $\{-\tau\}\times V$. This gluing is smooth because of the presence of the contact vector field $Z$ near $V$.

Similarly, we can define the completion $\hat{X}$ of $X$ as a gluing
\begin{equation}
\hat{X}=M_-\cup_{V_-}\mathbb{R}_{\geq 0}\times V,
\end{equation}
here identifying $V_-$ with $V\times \{0\}$. This completion is naturally a Liouville manifold with primitive $\hat{\theta}$ given by $\hat{\theta}|_{M_-}=\theta$ and $\hat{\theta}|_{\mathbb{R}_{\geq 0}\times V}=e^t\beta$. The flow of $Z$ gives a symplectomorphism
\begin{equation}
    M\setminus L\cong M_-\cup_{V_-}\mathbb{R}_{\leq 0}\times V,
\end{equation}
where $V_-$ is identified with $\{0\}\times V$.

We can choose an almost-complex structure $J_V$ on $V\times\mathbb{R}_t$ such that
\begin{itemize}
    \item $J_V$ is $d(e^t\beta)$-compatible
    \item $\beta\circ J_V=dt$
    \item $\mathcal{L}_{\partial_t}J_V=0$.
\end{itemize}
For some $\epsilon>0$, the vector field $Z$ gives us a symplectomorphism from a neighbourhood of $V$ in $M$ to $((-\epsilon,\epsilon)_t\times V,d(e^t\beta))$. Fix an almost-complex structure $J_+$ on $M_+$ such that
\begin{itemize}
    \item $J_+$ is $\omega$-compatible,
    \item $J_+|_{[0,\epsilon)\times V}=J_V$
    \item $J_+$ is standard on $\iota(B_r)$, i.e. $\iota^*J_+=i$.
\end{itemize}

Now consider the space of $\omega$-compatible almost-complex structures $J_-$ on $M_-$ satisfying
\begin{equation}
J_-|_{(-\epsilon,0]\times V}=J_V.
\label{standard_J_-}
\end{equation}
Given any such $J_-$, we obtain a unique almost-complex structure $J_\tau$ on $M_\tau$ such that
\begin{itemize}
    \item $J_\tau|_{M_\pm}=J_\pm$,
    \item $J_\tau|_{[-\tau,0]\times V}=J_V$.
\end{itemize}
and similarly a limiting family of almost-complex structures $J_\infty$ on $M_\infty$. The family of almost-complex manifolds $(M_\tau,J_\tau)$ and the limiting building $(M_\infty,J_\infty)$ agrees with the splitting construction described in Section 3.4 of \cite{Bou}. 

By hypothesis, all Reeb orbits in $(V,\beta)$ are non-degenerate. By Theorem 7.2 of \cite{Wendl}, there is a dense subset of those $\omega$-compatible $J_-$ satisfying \eqref{standard_J_-} such that any asymptotically cylindrical $J_\infty$-holomorphic curve in $\hat{X}$ which has an injective point in $M_-$ is Fredholm-regular. Similarly, we can ensure that any such curve is a regular point for $ev$, as in Section 4.4 of \cite{Wendl2}. Furthermore, any $J_\infty$-holomorphic curve in $\hat{X}$ will be a multiple cover of a curve with an injective point $M_-$.

As described in Section 2.7 of \cite{CM}, We can define a family of symplectic forms $\omega_\tau$ for each $\tau\geq 0$ by letting $\omega_\tau|_{M_+}=\omega|_{M_+}$, $\omega_\tau|_{[-\tau,0]_t\times V}=d(e^t\beta)$, and $\omega_\tau|_{M_-}=e^{-\tau}\omega|_{M_-}$. It follows that $J_\tau$ is $\omega_\tau$-compatible for each $\tau$. For each $\tau\geq 0$, $(M_\tau,\omega_\tau)$ is symplectomorphic to $(M,\omega)$ for each $\tau$. We therefore obtain a genus-$0$ $J_\tau$-holomorphic curve $u_\tau$ in $M_\tau$ with homology class $[u_\tau]=A$, with two marked points such that $ev(u_\tau)=(x,y)$. By Theorem 10.3 of \cite{Bou}, some subsequence of $(u_\tau)$ converges to a split holomorphic building $u$ in $(M_\infty,J_\infty)$, of genus $0$ with two marked points such that $ev(u)=(x,y)$, with total homology class $A$. The components $u_\pm$ have the required properties by our choice of $J_\pm$.
\end{proof}
\section{Virtual Dimension Formula}
We now collect all of our previous results to establish properties of the asymptotes of the components of our holomorphic building $u$.
\begin{lem}
\label{orbit_properties_estimates}
There exists a function $f:(-\infty,0]\rightarrow Y^R$ such that the following holds.

Let $u$ be a holomorphic building as in Lemma \ref{neck_stretch}. Let $\tilde{\gamma}_i$ for $i=1,...,k$ be the Reeb orbits to which the positive punctures of some component of $u$ are asymptotic. We view these as Reeb orbits in the perturbed contact hypersurface $(Y^R,e^f\alpha)$ via the map $p_f$, after rescaling time by a factor of $\sigma$. 

Then there exist Reeb orbits $\gamma_i$ of period $S_i$ in the unperturbed hypersurface $(Y^R,\alpha)$ for $i=1,...,k$, a trivialisation $\tilde{\tau}_{out}$ of $\tilde{\gamma}_i^*\xi$, and a disk $u_{out}(\tilde{\gamma}_i)\in\pi_2(M,\tilde{\gamma}_i)$ bounding $\tilde{\gamma}_i$ for each $i$, satisfying the following properties:
\begin{enumerate}
    \item $\tilde{\gamma}_i$ is non-degenerate,
    \item $\mu_{CZ}^{\tilde{\tau}_{out}}(\tilde{\gamma}_i)\leq -2S_i\sum_{v\in I(\gamma_i)}a_v(\gamma_i)w_v+n-1$
    \item $2c_1^{\tilde{\tau}_{out}}(u_{out}(\tilde{\gamma}_i))=0$,
    \item $\pmb{\kappa}(u_{out}(\tilde{\gamma}_i))=S_i\sum_{v\in I(\gamma_i)}a_v(\gamma_i)\lambda_v$,
\item $\omega(u_{out}(\tilde{\gamma}_i))<C'\epsilon(R)\pmb{\kappa}(u_{out}(\tilde{\gamma}_i))$,
\end{enumerate}
where $C'>0$ is a constant depending only on $\kappa_{\min}$.
\end{lem}
\begin{proof}
Let $\tilde{S}_i$ denote the period of $\tilde{\gamma}_i$. Recall that $A\in H_2(M;\mathbb{Z})$, the total homology class of $u$, was fixed. To each component of $u$, we can associate a positive number called the $\omega$-energy (see Section 5.3 of \cite{Bou}). It follows from positivity of $\omega$-energy of all the components that
\begin{equation}
    \omega(A)-\sigma\sum_{i=1}^k\tilde{S}_i\geq 0,
    \label{omega_energy_estimate}
\end{equation}
so $\tilde{S}_i\leq \sigma^{-1}\omega(A)$ for all $i$. 

Let $\epsilon=\frac{\epsilon(R)}{\kappa_{\min}}$. For any $\delta>0$, we can choose $f$ such that all Reeb orbits in $(Y,e^f\alpha)$ are non-degenerate (so condition 1 is satisfied) and $||f||_{C^\infty}<\delta$. We can therefore apply Lemma \ref{bounded_period_perturbation} once and for all with $S_{\max}>\sigma^{-1}\omega(A)$, so we may assume that $\tilde{\gamma}_i$ is an $\epsilon$-perturbation of some $S_i$-periodic Reeb orbit $\gamma_i$ in $(Y^R,\alpha)$ for each $i$. For each $i$, let $\Sigma_i$ denote the cylinder joining $\tilde{\gamma}_i$ to $\gamma_i$, and $\tau$,$\tilde{\tau}$ the trivialisations of $\gamma_i^*\xi$,$\tilde{\gamma}_i^*\xi$ respectively, witnessing this fact. 

Let $u_{out}(\tilde{\gamma}_i)$ denote the class in $\pi_2(M,\tilde{\gamma}_i)$ obtained by attaching a representative of $u_{out}(\gamma_i)$ to $\Sigma_i$ along $\gamma_i$. We obtain a trivialisation $\tilde{\tau}_{out}$ for $\tilde{\gamma}_i^*\xi$ by letting $\tilde{\tau}_{out}(\tilde{S}_it)=\tau_{out}(S_it)\circ\tau(S_it)^{-1}\circ\tilde{\tau}(\tilde{S}_it)$. In other words, we extend the trivialisation $\tau_{out}$ for $\gamma_i^*\xi$ over $\Sigma_i$ to a trivialisation for $\tilde{\gamma}_i^*\xi$.

It follows from the fact that $\tilde{\gamma}_i$ is an $\epsilon$-perturbation of $\gamma_i$ that $c_1^{\tilde{\tau}_{out}}(u_{out}(\tilde{\gamma}_i))=c_1^{\tau_{out}}(u_{out}(\gamma_i))$ and $\mu_{CZ}^{\tilde{\tau}_{out}}(\tilde{\gamma}_i)-\mu_{CZ}^{\tau_{out}}(\gamma_i)\in [0,\dim\ker(D\phi^1)_{\gamma_i(0)}|_\xi-id)]$. Conditions 2 and 3 on the relative Chern numbers and Conley-Zehnder indices follow from the construction of the outer caps and Lemma \ref{index_estimate}.

Condition 4 follows from Lemma \ref{outer_caps_action} and the fact that $u_{out}(\gamma_i)$ and $u_{out}(\tilde{\gamma}_i)$ represent the same class in $H_2(M,X;\mathbb{Z})$. 

To show that Condition 5 holds, let $C>0$ be the constant from Lemma \ref{outer_caps_action} and note that $\pmb{\kappa}(u_{out}(\tilde{\gamma}_i))\geq \kappa_{\min}$, so $|\tilde{S}_i-S_i|\leq \epsilon(R)[\omega,\theta](u_{out}(\tilde{\gamma}_i))$. Now take $C'=C+1$.
\end{proof} 

\begin{lem}
Let $\tilde{\gamma}_i$ for $i=1,...,k$ be the asymptotes of $\tilde{u}_-$, and $\gamma_i$ the corresponding Reeb orbits in $(Y^R,\alpha)$ of period $S_i$, and $\tau_{out}$ the trivialisation of $\tilde{\gamma}_i^*\xi$ obtained from Lemma \ref{orbit_properties_estimates}. These satisfy
\label{total_c1}
    \begin{equation}
        2c_1^{\tau_{out}}(\tilde{u}_-)=\sum_{i=1}^kS_i\sum_{v\in I({\gamma}_i)}a_v(\gamma_i)\lambda_v.
    \end{equation}
\end{lem}
\begin{proof}
By gluing $u_{out}(\tilde{\gamma}_i)$ to $\tilde{u}_-$ along $\tilde{\gamma}_i$ for each $i$, we obtain a class $\Sigma\in \pi_2(M)$, which satisfies $c_1(\Sigma)=c_1^{\tau_{out}}(\tilde{u}_-)$ by additivity of relative Chern numbers, and the fact that $c_1^{\tau_{out}}(u_{out}(\tilde{\gamma}_i))=0$ for each $i$.

Because $c_1$ is supported on $V$, 
\begin{equation}
2c_1(\Sigma)=\sum_{i=1}^k\pmb{\lambda}(u_{out}(\tilde{\gamma}_i)).
\end{equation}
The result follows from substituting the properties of the caps $u_{out}(\tilde{\gamma}_i)$ from Lemma \ref{orbit_properties_estimates}.
\end{proof}

\begin{lem}
\label{k_estimate}
Suppose $f$ is chosen as in Lemma \ref{orbit_properties_estimates}. Let $u$ be a holomorphic building as in Lemma \ref{neck_stretch}. Let $k$ denote the number of positive punctures of $\tilde{u}_-$.
If $\tilde{\sigma}_{crit}\leq 0$, then
\begin{equation}
    k\geq \lceil 2-\tilde{\sigma}_{crit}N_M\rceil.
\end{equation}
\end{lem}
\begin{proof}
By the virtual dimension formula  (see Theorem 7.1 of \cite{Wendl}), $\tilde{u}_-$ is an element of a moduli space which has virtual dimension given by 
\begin{equation}
ind(\tilde{u}_-)=(n-3)(2-k)+2c_1^{\tilde{\tau}_{out}}(\tilde{u}_-)+\sum_{i=1}^{k}\mu_{CZ}^{\tilde{\tau}_{out}}(\tilde{\gamma_i})+2-2n.
\end{equation}
Because $\tilde{u}_-$ is Fredholm regular by assumption, the virtual dimension is realised and $ind(\tilde{u}_-)\geq 0$, 
by Lemma \ref{total_c1} and Property 2 from Lemma \ref{orbit_properties_estimates}, we have 
\begin{equation}
0\leq 2k-4+\sum_{i=1}^kS_i\sum_{v\in I(\gamma_i)}a_v(\gamma_i)(\lambda_v-2w_v).
\end{equation}

By definition of $\tilde{\sigma}_{crit}$,
\begin{equation}
0\leq 2k-4+\tilde{\sigma}_{crit}\sum_{i=1}^{k}S_i\sum_{v\in I(\gamma_i)}\lambda_v a_v(\gamma_i).
\end{equation}
By the properties of the caps $u_{out}(\tilde{\gamma}_i)$ from Lemma \ref{orbit_properties_estimates},
\begin{equation}
\label{intermediate_estimate}
0\leq 2k-4+\tilde{\sigma}_{crit}\pmb{\lambda}\left(\sum_{i=1}^ku_{out}(\tilde{\gamma}_i)\right).
\end{equation}
Recall that the class $\Sigma\in \pi_2(M)$ from the proof of Lemma \ref{total_c1} satisfies \begin{equation}
2c_1(\Sigma)=\pmb{\lambda}\left(\sum_{i=1}^ku_{out}(\tilde{\gamma}_i)\right)>0.
\end{equation}
Substituting this inequality into Equation \ref{intermediate_estimate}, and noting that $N_M\leq c_1(\Sigma)$ gives the required result.
\end{proof}

\section{Areas of $J_\infty$-Holomorphic Curves}

\begin{lem}
\label{area_bound_simple}
If $u_-$ has at least $k$ positive punctures, then 
\begin{equation}
\omega(u_+)\leq \omega(A)-(k-1)\kappa_{\min}.
\end{equation}
\end{lem}
\begin{proof}
By genus considerations, the building $u$ has at least $k$ connected components in $M\setminus L$, including $u_+$, which we denote by $u_+=u_0,u_1,...,u_{k-1},...$. For each $i$, since each component has positive area, and $\omega$ is supported on $V$, $\omega(u_i)\geq \kappa_{\min}$. The result follows from the fact that $\omega(A)\geq \sum_{i=0}^{k-1}\omega(u_i)$.
\end{proof}
\begin{lem}
If $u_-$ has at least $k$ positive punctures, then
\begin{equation}
\omega(u_+\cap \{\rho^0\geq \sigma\})\leq \omega(A)-2\sigma N_M-(1-\sigma)(k-1)\kappa_{\min}+C''\epsilon(R)
\end{equation}
where $C''>0$ is a constant depending only on $M$, $\kappa_{\min}$ and $\sigma$.
\label{area_bound_advanced}
\end{lem}
\begin{proof}
Let $\tilde{\gamma}_j$ for $j=1,...,\ell$ denote the asymptotes of $u_+$, and $\tilde{S}_j$ the period of $\tilde{\gamma}_j$ (viewed as a Reeb orbit in $(Y^R,e^f\alpha)$) for each $j$.

Recall that the hypersurface $V$ partitions $M\setminus L$ into a region containing $\{\rho^R\geq \sigma\}$, and a region symplectomorphic to $(V\times (\mathbb{R}_{\geq 0})_t,d(e^t\beta))$.
Since $J_\infty$ is of contact type in this region we can equip $V\times \mathbb{R}_{\geq 0}$ with metric $g$ such that $(d(e^t\beta),J_\infty,g)$ is a compatible triple, and for any $\nu\in TV$, $d\beta(\nu,J_\infty\nu)=||\nu^{\parallel}||^2\geq 0$, where $\nu^{\parallel}$ is the orthogonal projection onto $\ker\beta$. This implies that the integral of $d\beta$ over any subset of $u_0$ is non-negative (i.e. the same reason that $\omega$-energy is positive), so
    \begin{equation}
        \int_{u_+\cap V}\beta-\sigma\sum_{j=1}^\ell\tilde{S}_j=\int_{u_+\cap V\times \mathbb{R}_{\geq 0}}d\beta\geq 0.
    \end{equation}
Furthermore,
\begin{equation}
    \omega(u_+\cap \{\rho^0\geq \sigma\})\leq \omega(u_+)-\int_{u_+\cap V}\beta\leq \omega(u_+)-\sigma\sum_{j=1}^\ell\tilde{S}_j.
    \label{symp_monotonicity}
\end{equation}

Topologically, we obtain a class $\Sigma\in\pi_2(M)$ by cutting out $u_+$ from $\overline{u}$, and gluing $u_{out}(\tilde{\gamma})$ along $\tilde{\gamma}$ for every puncture of $u_+$ asymptotic to a Reeb orbit $\tilde{\gamma}$. The area of $\Sigma$ is given by
\begin{equation}
\omega(\Sigma)=\sum_{i=1}^{k-1}\omega(u_i)+\sum_{j=1}^\ell(\omega(u_{out}(\tilde{\gamma}_j))+\tilde{S}_j).
\end{equation}

Substituting this into Equation \eqref{symp_monotonicity},
\begin{equation}
\omega(u_+\cap \{\rho^0\geq \sigma\})\leq \omega(u_+)-\sigma\left(\omega(\Sigma)-\sum_{i=1}^{k-1}\omega(u_i)-\sum_{j=1}^\ell\omega(u_{out}(\tilde{\gamma}_j))\right).   
\end{equation}

By the area bounds from Lemma \ref{orbit_properties_estimates},
\begin{equation}
\omega(u_+\cap \{\rho^0\geq \sigma\})\leq \omega(u_+)-\sigma(1-C'\epsilon(R))\left(\omega(\Sigma)-\sum_{i=1}^{k-1}\omega(u_i)\right).  
\end{equation}
Now, noting that $\omega(\Sigma)\geq 2N_M$,
\begin{equation}
\omega(u_+\cap \{\rho^0\geq \sigma\})\leq \sigma\omega(A)+(1-\sigma)\omega(u_+)-2N_M\sigma(1-C'\epsilon(R)). 
\end{equation}
Now let $C''=2N_M\sigma C'$. The result follows from Lemma \ref{area_bound_simple}.
\end{proof}

\subsection{Proof of Main Theorems}
We now collect all of our previous results into a single theorem.
\begin{thm}
\label{T:1}
\label{main_prop}
Suppose $A\in H_2(M)$ satisfies Hypothesis \ref{rat_con}.

If $\tilde{\sigma}_{crit}\leq 0$, then
\begin{equation}
W_G(\{\rho^0>\sigma\})\leq \omega(A)-2\sigma N_M -(1-\sigma)\lceil1-\tilde{\sigma}_{crit}N_M\rceil\lambda_{\min}.
\end{equation}
\end{thm}
\begin{proof}
Let $r<\sqrt{\pi^{-1}W_G(\{\rho^0\geq \sigma\})}$. There exists a symplectic embedding $\iota:B_r\rightarrow \{\rho^0> \sigma\}$. By the neck-stretching argument described in Lemma \ref{neck_stretch}, we obtain a $J_\infty$-holomorphic curve $u_+$ in $M\setminus L$, such that $\iota^{-1}(u_+)$ is a holomorphic subvariety of $B_r$. By the Lelong inequality,
\begin{equation}
    1\leq \frac{1}{\pi r^2}\int_{\iota^{-1}(u_+)}\omega_{st}
\end{equation}
where $\omega_{st}$ denotes the standard symplectic form on $\mathbb{C}^n$. By positivity of area, $\omega(u_+\cap \{\rho^0\geq \sigma\})\geq \omega(u_+\cap \iota(B_r))\geq \pi r^2$. The result follows from combining this inequality with Lemma \ref{k_estimate} and Lemma \ref{area_bound_advanced} (where we can ignore the error term by letting $R\to 0$), then letting $r\to \sqrt{\pi^{-1}W_G(\{\rho^0\geq \sigma\})}$.
\end{proof}

Now we specialise to the case of projective space.
\begin{thm}
\label{T:2}
\label{proj_skel}
    Take $M=\mathbb{CP}^n$, and $\omega$ the Fubini-Study form , normalised so that $\omega(\mathbb{CP}^1)=1$. If $\tilde{\sigma}_{crit}\leq 0$, then $L\subset \mathbb{CP}^n$ is a Lagrangian barrier. Moreover,
    \begin{equation}
        W_G(\mathbb{CP}^n\setminus L)\leq 1-\lceil 1-\tilde{\sigma}_{crit}(n+1)\rceil \kappa_{\min}.
    \end{equation}
\end{thm}
\begin{proof}
We may take $A$ to be the class of a line. Note that $W_G(\mathbb{CP}^n)=\omega(\mathbb{CP}^1)=2(n+1)$ with respect to our normalisation convention, so $\omega(A)=2c_1(A)=2N_{\mathbb{CP}^n}=2(n+1)$, and apply Theorem \ref{main_prop}.
\end{proof}

\section{Ball embeddings from pairs of orthogonal curves}
\label{lower_bounds}
\begin{lem}
    Suppose $n=2$, and for some $u\in V$, $\{u\in V| u<v\}=\{v_1,v_2\}$ for some $v_1\neq v_2$.
    
    Suppose $\mathring{D}_{u}=\{p\}$, and for $i=1,2$, $D_i=\mathring{D}_{v_i}\cup \{p\}$ is a $2$-dimensional symplectic submanifold. Suppose further that for $i=1,2$, $r_{v_i}$ extends over a neighbourhood of $p$, in which $r_u=r_{v_1}+r_{v_2}$ and $\{r_{v_1},r_{v_2}\}=0$.

    In other words, $D_1\cup D_2$ is an orthogonal symplectic crossings divisor, and the radial Hamiltonians are constructed in a standard way. Equivalently, near the point $p$, $D_1\cup D_2$ is symplectically modelled on $0\in\{z_1z_2=0\}\subset \mathbb{C}_z^2$, and the radial Hamiltonians generate the standard torus action.    
    Suppose
    \begin{equation}
        \pi r^2<\min\{\omega(D_1),\omega(D_2),\kappa_1,\kappa_2\}.
    \end{equation}
    Then there exists a symplectic embedding $\iota:B_r\rightarrow M\setminus L$.
\end{lem}

\begin{proof}    
    We fix contractible open sets $C_i\subset D_i$ for $i=1,2$ containing $p$, such that $\omega(C_i)=\pi r^2$. Our local model near $p$ gives us a symplectic embedding $\iota:B_\epsilon\rightarrow M$ for some small $\epsilon>0$, such that $\iota(0)=p$, $r_{u}\circ \iota$ generates the diagonal $U(1)$-action, and for $i=1,2$, $\iota(\{z_i=0\}\cap B_\epsilon)\subset C_i$ and $r_{v_i}\circ \iota$ generates the standard $U(1)$-action on the $i$-th factor.

    For $i=1,2$, we can extend $\iota|_{\{z_i=0\}\cap B_\epsilon}$ to a symplectomorphism $\iota_i:\{z_i=0\}\cap B_r\rightarrow C_i$, by our assumption on $r$. We can choose an $\omega$-compatible metric $g$ on a neighbourhood of $C_i$ such that $\iota^*g$ and $\iota_i^*g$ are standard, and $g$ is invariant under the circle action generated by $r_{v_i}$. 
    
    The normal bundle for $C_i\subset M$ is symplectomorphic to $B_r\cap\{z_i=0\}\times \mathbb{C}$. By applying an equivariant Moser argument to the exponential map, we can extend $\iota$ to a neighbourhood $U_i$ of $\{z_i=0\}\cap B_r$, such that $\iota|_{\{z_i=0\}\cap B_r}=\iota_i$, and $r_{v_i}\circ\iota$ still generates the standard circle action on the $i$-th factor, where defined. We may assume $U_1\cap U_2\subset B_\epsilon$, so we can extend $\iota$ to a neighbourhood $U_1\cup U_2$ of $\{z_1z_2=0\}\cap B_r$.

    We can extend the functions $r_{v_i}$ to $\iota(U_1\cup U_2)$ for each $i$ in the standard way. Possibly after restricting to smaller neighbourhoods $U_i$ of $\{z_i=0\}\cap B_r$ for $i=1,2$, we may assume that $\theta$ is invariant under the torus action generated by $(r_{v_1},r_{v_2})$ on $\iota(U_1\cup U_2)$. Via the Liouville flow, we can extend $(r_{v_1},r_{v_2})$ along flowlines of $Z$ subject to the differential equation $dr_{v_i}(Z)=\kappa_{v_i}-r_{v_i}$. By our assumption on $r$, the image of the moment map $(r_{v_1},r_{v_2})$ contains the set $\Delta=\{0\leq r_{v_1},0\leq r_{v_2},r_{v_1}+r_{v_2}\leq \pi r^2\}$, so we can use the Lagrangian torus fibration $(r_{v_1},r_{v_2})$ to extend the domain and range of $\iota$ to obtain a symplectomorphism $\iota:B_r\rightarrow \Delta$.
\end{proof}

\begin{cor}
    In the case where $n=2$, and two of the hyperplanes $H_i$ are the coordinate hyperplanes, the bound on $W_G(M\setminus L)$ in Corollary \ref{generic_hyperplanes} is tight.
\end{cor}

\section*{Acknowledgements}
The author thanks Jonny Evans for many informative conversations and suggestions during the preparation of this paper.

For the purpose of open access, the author has applied a Creative Commons Attribution (CC BY) licence to any Author Accepted Manuscript version arising.
\printbibliography

@article{Gro,
author = {Gromov, M.},
journal = {Inventiones mathematicae},
pages = {307-348},
title = {Pseudo holomorphic curves in symplectic manifolds.},
url = {http://eudml.org/doc/143289},
volume = {82},
year = {1985},
}

@misc{Wendl,
      title={Lectures on Symplectic Field Theory}, 
      author={Chris Wendl},
      year={2016},
      eprint={1612.01009},
      archivePrefix={arXiv},
      primaryClass={math.SG},
      url={https://arxiv.org/abs/1612.01009}, 
}

@misc{Lu,
      title={An extension of Biran's Lagrangian barrier theorem}, 
      author={Guangcun Lu},
      year={2004},
      eprint={math/0111183},
      archivePrefix={arXiv},
      primaryClass={math.SG},
      url={https://arxiv.org/abs/math/0111183}, 
}

@article{CM,
author = {Cieliebak, Kai and Mohnke, Klaus},
year = {2005},
month = {12},
pages = {},
title = {Compactness for punctured holomorphic curves},
volume = {3},
journal = {Journal of Symplectic Geometry},
doi = {10.4310/JSG.2005.v3.n4.a5}
}

@article{Biran,
  title={Lagrangian barriers and symplectic embeddings
},
  author={Paul Biran},
  journal={Geometric \& Functional Analysis GAFA},
  year={2001},
  volume={11},
  pages={407-464},
  url={https://api.semanticscholar.org/CorpusID:122402893}
}

@Article{GS,
title = {Convexity properties of the moment mapping},
journal = {Inventiones Mathematicae},
volume = {67},
pages = {491-513},
year = {1982},
author = {Victor Guillemin and Shlomo Sternberg}
}

@article{Bou,
   title={Compactness results in Symplectic Field Theory},
   volume={7},
   ISSN={1465-3060},
   url={http://dx.doi.org/10.2140/gt.2003.7.799},
   DOI={10.2140/gt.2003.7.799},
   number={2},
   journal={Geometry \&; Topology},
   publisher={Mathematical Sciences Publishers},
   author={Bourgeois, Frederic and Eliashberg, Yakov and Hofer, Helmut and Wysocki, Kris and Zehnder, Eduard},
   year={2003},
   month=dec, pages={799–888} }

@misc{Ga,
      title={Superheavy Skeleta for non-Normal Crossings Divisors}, 
      author={Elliot Gathercole},
      year={2024},
      eprint={2408.13187},
      archivePrefix={arXiv},
      primaryClass={math.SG},
      url={https://arxiv.org/abs/2408.13187}, 
}

@misc{Long_iteration,
      title={Index iteration theory for symplectic paths with applications to nonlinear Hamiltonian systems}, 
      author={Yiming Long},
      year={2003},
      eprint={math/0304265},
      archivePrefix={arXiv},
      primaryClass={math.DG},
      url={https://arxiv.org/abs/math/0304265}, 
}

@article{Long_index,
author = {Long, Yiming},
year = {1997},
month = {09},
pages = {},
title = {A Maslov-type index theory for symplectic paths},
volume = {10},
journal = {Topological Methods in Nonlinear Analysis},
doi = {10.12775/TMNA.1997.021}
}

@misc{Wendl2,
      title={Lectures on Holomorphic Curves in Symplectic and Contact Geometry}, 
      author={Chris Wendl},
      year={2014},
      eprint={1011.1690},
      archivePrefix={arXiv},
      primaryClass={math.SG},
      url={https://arxiv.org/abs/1011.1690}, 
}

@misc{TMZ,
      title={Normal Crossings Singularities for Symplectic Topology: Structures}, 
      author={Mohammad Farajzadeh Tehrani and Mark McLean and Aleksey Zinger},
      year={2021},
      eprint={2112.13125},
      archivePrefix={arXiv},
      primaryClass={math.SG},
      url={https://arxiv.org/abs/2112.13125}, 
}

@misc{Bre,
      title={Pinwheels as Lagrangian barriers}, 
      author={Joé Brendel and Felix Schlenk},
      year={2022},
      eprint={2210.00280},
      archivePrefix={arXiv},
      primaryClass={math.SG},
      url={https://arxiv.org/abs/2210.00280}, 
}

@article{Tian_2012,
   title={Symplectic geometry of rationally connected threefolds},
   volume={161},
   ISSN={0012-7094},
   url={http://dx.doi.org/10.1215/00127094-1548398},
   DOI={10.1215/00127094-1548398},
   number={5},
   journal={Duke Mathematical Journal},
   publisher={Duke University Press},
   author={Tian, Zhiyu},
   year={2012},
   month=apr }
\end{document}